\documentclass[twoside,9pt]{article}
\usepackage{graphicx,amscd,amsfonts,amsmath,amsthm,amssymb,latexsym,amsfonts,color}
\usepackage{epsfig,float,epstopdf,array,bm,cite}
\usepackage[bookmarksnumbered, colorlinks,plainpages]{hyperref}
\newtheorem{theorem}{\bf Theorem}
\newtheorem{lemma}{\bf Lemma}

\newcommand{\norm}[1]{{ \left\| { #1 } \right\| }}
\def\h{\hspace{-0.2cm}}
\begin{document}
\title{ A preconditioner based on the shift-splitting method for generalized saddle point problems}
\author{{ Davod Khojasteh Salkuyeh\footnote{Corresponding author}, Mohsen Masoudi and Davod Hezari}\\[2mm]
\textit{{\small Faculty of Mathematical Sciences, University of Guilan, Rasht, Iran}} \\
\textit{{\small E-mails: khojasteh@guilan.ac.ir, masoodi\_mohsen@ymail.com, d.hezari@gmail.com}\textit{}}}
\date{}
\maketitle
{\bf Abstract.} In this paper, we propose a preconditioner based on the shift-splitting method for generalized saddle point problems with nonsymmetric positive definite (1,1)-block and symmetric positive semidefinite $(2,2)$-block. The proposed preconditioner is obtained from an basic iterative method which is unconditionally convergent.  We also present a relaxed version of the proposed method. Some numerical experiments are presented to show the effectiveness of the method.\\[-3mm]

%\noindent
%\noindent{\bf------------------------------------------------------------------------------------}

\pagestyle{myheadings}\markboth{D.K. Salkuyeh, M. Masoudi and D. Hezari}{A preconditioner based on the shift-splitting method for generalized saddle point problems}

\thispagestyle{empty}

\section{Introduction}

We consider the solution of the following large and sparse generalized saddle point problem
\begin{equation}\label{sdpr}
\mathcal{A}u=\left(
  \begin{array}{cc}
    A & B^T \\
   -B & C
  \end{array}
\right)
\left(
  \begin{array}{c}
    x \\
    y
  \end{array}
\right)=
\left(
  \begin{array}{c}
    f \\
    -g
  \end{array}
\right)=b,
\end{equation}
where $A\in \mathbb{R}^{n \times n}$ is nonsymmetric positive definite ($x^TAx>0$ for all $0\neq x \in \Bbb{R}^n$), $C\in \mathbb{R}^{m \times m}$ is symmetric positive semidefinite, the matrix $B\in \mathbb{R}^{m\times n}$ is of  full row rank, $x,f\in \mathbb{R}^{n}$, $y,g\in \mathbb{R}^{m}$ and $m\leq n$. It can be verified that the system (\ref{sdpr}) has a unique solution \cite[Lemma 1.1]{Benzi-a1}. Saddle point problems of the form (\ref{sdpr}) arise from finite difference or finite element discretization of the Navier-Stokes problem (see \cite{Benzi-sur} and references therein). Several iterative method have been presented to solve system (\ref{sdpr}) or some special cases of it in the literature. The main methods have been reviewed in \cite{Benzi-sur}. In \cite{Benzi-a1}, Benzi and Golub presented the Hermitian and skew-Hermitian splitting  (HSS) method to solve (\ref{sdpr}). Since, in general, the HSS method is too slow to be used to solve \cite{Benzi-a1}, they used the GMRES method in conjunction with the  preconditioner extracted from the HSS method to solve (\ref{sdpr}). Recently, when the matrix $A$ is symmetric positive definite,  Salkuyeh et al. in \cite{Salkuyeh2015} have presented a stationary iterative method based on the shift-splitting method to solve (\ref{sdpr}). The proposed method naturally serves a preconditioner for the problem (\ref{sdpr}). More recently, Cao et al. in \cite{Cao2015} have considered the same iterative method to solve the system (\ref{sdpr}) when $C=0$. In this paper, we consider the problem (\ref{sdpr}) in its general form and investigate the convergence properties of the proposed iterative method and the corresponding preconditioner.

\section{Main results}

For the sake of the simplicity we use the nations used in \cite{Salkuyeh2015}.
Assuming  $\alpha , \beta >0$, Salkuyeh et al. in \cite{Salkuyeh2015} considered the splitting  $\mathcal{A}=\mathcal{M}_{\alpha,\beta}-\mathcal{N}_{\alpha,\beta}$ for the saddle point problem (\ref{sdpr}) with $A$ being symmetric positive definite, where
\begin{equation}\label{MGHSSsplit}
\mathcal{M}_{\alpha,\beta}=
\frac{1}{2}
\left( {\begin{array}{*{20}{c}}
\alpha I +A & B^T\\
-B & \beta I+ C
\end{array}} \right) \quad \textrm{and} \quad
\mathcal{N}_{\alpha,\beta}=\frac{1}{2}
\left( {\begin{array}{*{20}{c}}
\alpha I -A & -B^T\\
B & \beta I- C
\end{array}} \right).
\end{equation}
This splitting gives the following basic iterative method (hereafter is denoted by the MGSS iteration scheme)
\begin{equation}\label{MGHSSit}
\mathcal{M}_{\alpha,\beta} u^{(k+1)}=\mathcal{N}_{\alpha,\beta} u^{(k)}+b
\end{equation}
for solving the linear system \eqref{sdpr}, where $u^{(0)}$ is an initial guess. In continuation, we show that the symmetry of the matrix $A$ can be omitted. From Eq. (\ref{MGHSSit}), we see that the iteration matrix of the proposed method is $\Gamma_{\alpha,\beta} =\mathcal{M}_{\alpha,\beta}^{-1}\mathcal{N}_{\alpha,\beta}$. Hence, the method is convergent if and only if the spectral radius of $\Gamma_{\alpha,\beta}$ is less than 1, i.e., $\rho(\Gamma_{\alpha,\beta})<1$.

%%%%%%%%%%%%%%%%%%%%%%%% Lemma 1%%%%%%%%%%%%%%%%%%%%%%%%%%%%%%%%%%%%%%%%%%%%%%%%%%%%%%%%%%%%%%%%%%%%%%%

\begin{lemma}\label{Lem1}(\cite[Lemma 1]{Salkuyeh2015})
Assume that $\alpha$ and $\beta$ are two positive numbers. If $\lambda$ is an eigenvalue of the matrix  $\Gamma_{\alpha,\beta}$, then $\lambda \neq \pm1$.
\end{lemma}

\begin{lemma}\label{Lem2}
Let $A\in \Bbb{R}^n$ be a nonsymmetric positive definite matrix. Then, $\Re(x^*Ax)>0$,  for any $0\neq x \in \Bbb{C}^n$.
\begin{proof}
Let $x=r+is$, where $r,s\in \Bbb{R}^n$. Obviously, both of the vectors $r$ and $s$ can not be zero simultaneously. On the other hand,
\[
x^*Ax=(r^T-is^T)A(r+is)=r^TAr+s^TAs+ir^T (A-A^T)s.
\]
Hence, $\Re(x^*Ax)=r^TAr+s^TAs>0$.
\end{proof}
\end{lemma}

%%%%%%%%%%%%%%%%%%%%%%%%% Therorem 1 %%%%%%%%%%%%%%%%%%%%%%%%%%%%%%%%%%%%%%%%%%%%%%%%%%%%%%%%%%%%%%%%%%

\begin{theorem}
\label{th1}
Let $\lambda$ be an eigenvalue of the matrix $\Gamma$ and $\alpha , \beta >0$. Then $| {\lambda}|<1$.
\end{theorem}
\begin{proof}

Let $u=(x ; y)$ be an eigenvector corresponding to the  eigenvalue $\lambda $ of $\Gamma_{\alpha,\beta}$.
Then, we have $\mathcal{N}_{\alpha,\beta}u=\lambda \mathcal{M}_{\alpha,\beta}u$ which is equivalent to
\begin{eqnarray}
  (\alpha I-A)x-B^Ty &\h=\h& \lambda (\alpha I+A)x+\lambda B^Ty, \label{EigEq1} \\
  Bx+(\beta I-C)y &\h=\h& -\lambda Bx+ \lambda (\beta I+C) y.  \label{EigEq2}
\end{eqnarray}
According to Theorem 1 in \cite{Salkuyeh2015} we have $x\neq 0$.

Without loss of generality it is assumed that $\norm{x}_2=1$. Premultiplying both sides of \eqref{EigEq1} by $x^*$ yields
\begin{equation}\label{EigEq3}
\alpha-x^*Ax -(Bx)^*y=\lambda (\alpha \norm{x}_2^2+x^*Ax)+\lambda (Bx)^*y.
\end{equation}
Since $A$ is positive definite, according to Lemma \ref{Lem2} we have $\Re(x^*Ax)>0$.
If $Bx=0$, then Eq. \eqref{EigEq3} implies
\begin{eqnarray*}
|\lambda| =\frac{|\alpha  -x^*Ax|}{|\alpha +x^*Ax|}=\frac{\sqrt{ (\alpha  -\Re(x^*Ax))^2+(\Im(x^*Ax))^2}}{\sqrt{ (\alpha  +\Re(x^*Ax))^2+(\Im(x^*Ax))^2}}<1.
 \end{eqnarray*}
We now assume that $Bx \neq 0$. In this case, from Eq. \eqref{EigEq2} we obtain
\begin{equation}\label{EigEq4}
 Bx=\frac{\beta (\lambda-1)}{\lambda+1}y+Cy.
 \end{equation}
Substituting Eq. \eqref{EigEq4} in \eqref{EigEq3} yields
 \begin{equation*}
 (1-\lambda) \alpha -(1+\lambda)x^*Ax=(1+\lambda)\left( \beta \frac{ \overline{\lambda}-1}{1+\overline{\lambda}}y^*y+y^*Cy\right).
 \end{equation*}
Letting $p=x^*Ax$, $q=y^*y$, and $r=y^*Cy$, it follows from the latter equation that
\begin{equation}\label{Eqwwb}
\alpha \omega + \beta q \overline{\omega}=p+r,\quad \textrm{with}\quad \omega=\frac{1-\lambda}{1+\lambda}.
\end{equation}
Since $\alpha,\beta,\Re(p)>0$ and $ q,r\geq 0$, form \eqref{Eqwwb} we see that
\[
\Re(w)=\frac{\Re(p)+r}{\alpha+\beta q}>0.
\]
Hence, we have
\[
|\lambda|=\frac{|1-\omega|}{|1+\omega|}=\sqrt { \frac{ (1-\Re(\omega))^2+\Im(\omega)^2 }{ (1+\Re(\omega))^2+\Im(\omega)^2 } }<1,
\]
which completes the proof.
\end{proof}

%%%%%%%%%%%%%%%%%%%%%%%%%%%%%%%%%%%% Remark 1 %%%%%%%%%%%%%%%%%%%%%%%%%%%%%%%%%%%%%%%%%%%%%%%%%%%%%%%%%%%%%

Theorem \ref{th1} shows that  the  MGSS method is convergent and therefore it provides the preconditioner $\mathcal{P}_{MGSS}=\mathcal{M}_{\alpha,\beta}$ for a Krylov subspace method such as GMRES, or its restarted version GMRES($m$) for solving the saddle point problem \eqref{sdpr}. Implementation of the method is  as described in \cite{Salkuyeh2015}. We can also use a relaxed version of the  MGSS (say RMGSS)  preconditioner
\[
\mathcal{P}_{RMGSS}=
\left( {\begin{array}{*{20}{c}}
A & B^T\\
-B & \beta I+ C
\end{array}} \right).
\]
for the saddle point  problem \eqref{sdpr}.  Similar to Theorem 2 in \cite{Salkuyeh2015} one may discuss about the eigenvalues distribution of the coefficient matrix of the preconditioned system.

\section{Numerical experiments}\label{SEC3}

We consider the steady-state Navier-Stokes equation
\[
\left\{
  \begin{array}{rl}
    -\nu \triangle {\bf u}+({\bf u}.\nabla){\bf u}+\nabla p&={\bf f},\\
    \nabla.{\bf u}&=0,
  \end{array}
\right.\quad {\rm in}\quad \Omega=[0,1]\times [0,1],
\]
where $\nu>0$. By the  IFISS package \cite{IFISS}, this problem is linearized by the Picard iteration and then discretized by using the stabilized Q1-P0 finite elements (see \cite{Elmanbook}). The stabilization parameter is set to be $0.1$. This yields a generalized saddle point problem of the form (\ref{sdpr}). The right-hand side vectors $f$ and $g$ are taken such that $x$ and $y$ are two vectors of all ones. In Table \ref{Tbl}, the generic properties of the coefficient matrix have been given.

We use GMRES(30) in conjunction with the preconditioner $\mathcal{P}_{MGSS}$ to solve the saddle point problem (\ref{sdpr}). Numerical results are given in Table \ref{Tbl}. In this table ``Iters" and ``CPU" stand, respectively, for the number of iterations  and the CPU time (in seconds) for the convergence. To show the effectiveness of the methods we also give the results of  GMRES(30) without preconditioning.  We use a null vector as an initial guess and the  stopping criterion $ \|b-{\cal A}x^{(k)}\|_2< 10^{-9}\| b\|_2$. In the implementation  of the preconditioner $\mathcal{P}_{MGSS}$ (see Algorithm 1 in \cite{Salkuyeh2015}),  we use the Cholesky factorization of $\beta I+C$ and the GMRES(10) method to solve the inner systems.  It is noted that, the inner iteration is terminated   when the residual norm is reduced by a factor of $10^2$ and the maximum of the inner iterations is set to be 40.  In the MGSS method the parameters $\alpha$ and $\beta$  are set to be $0.01$ and $0.001$, respectively. As seen, the proposed preconditioner is very effective in reducing the number of iterations and CPU times.

\begin{table}
\centering
\caption{Numerical results for the test problem with $\nu=1/50$. \label{Table2}}
\vspace{-0.cm}
\begin{tabular}{ccccccccccc}\\ \hline \\   %\cline{3-6}
               & && &  GMRES(30)\hspace{-1.5cm} & &  & &  MGSS \hspace{-1.5cm} &   \\
                   \cline{5-6} \cline{9-10}  \\[0mm]
Grid           & $m$ & $n$ &&  Iters   &  CPU(s)       &   &  & Iters  & CPU(s)  \\ \hline
$8 \times 8$   &  162   & 62   & &  59   & 0.08          &   &  & 6   & 0.90   \\[1mm]
$16 \times 16$ &  578   & 256  & & 115   & 0.28          &   &  & 8   & 2.02   \\[1mm]
$32 \times 32$ &  2178  & 1024 & & 608   & 4.95          &   &  & 12  & 3.58   \\[1mm]
$64 \times 64$ &  8450  & 4096 & & 3554   & 110.1        &   &  & 28  & 21.48   \\ \hline
\end{tabular}
\label{Tbl}
\end{table}

\enddocument